\documentclass[12pt,leqno]{amsart}
\pdfoutput=1
\usepackage[utf8]{inputenc}
\usepackage{amsmath,amsfonts,amssymb,bm,amsthm,mathptmx}
\usepackage[all,arc]{xy}
\usepackage{graphicx}

\usepackage{comment,float}

\usepackage[utf8]{inputenc}
\usepackage[dvipsnames]{xcolor}
\DeclareMathOperator{\diag}{diag}
\DeclareMathOperator{\sgn}{sgn}

\def\N0{\mathbb{N}_0}

\newtheorem{theorem}{Theorem}
\newtheorem{lemma}[theorem]{Lemma}
\newtheorem{proposition}[theorem]{Proposition}
\newtheorem{corollary}[theorem]{Corollary}

\newtheorem{open problem}{Open Problem}

\theoremstyle{definition}
\newtheorem*{remark}{Remark}
\newtheorem*{ack}{Acknowledgement}

\newtheorem{definition}[theorem]{Definition}

\def\sgn{\operatorname{sgn}}

\def\diag{\operatorname{diag}}

\def\Z{\mathbb Z}
\def\T{\textup{\textsf{T}}}
\def\G{,\text{G}}

\begin{document}
\title[Cyclotomic symmetrizable matrices]{Symmetrizable integer matrices having all their eigenvalues in the interval $[-2,2]$}
\author{James McKee}
\address{Department of Mathematics, Royal Holloway, University of London, Egham Hill, Egham, Surrey, TW20 0EX, U.K.}\email{james.mckee@rhul.ac.uk}
\author{Chris Smyth}
\address{School of Mathematics\\
University of Edinburgh\\
Edinburgh EH9 3FD\\
Scotland, U.K.}
\email{c.smyth@ed.ac.uk}

\maketitle
\begin{abstract}
The adjacency matrices of graphs form a special subset of the set of all integer symmetric matrices.
The description of which graphs have all their eigenvalues in the interval $[-2,2]$ (i.e., those having spectral radius at most $2$) has been known for several decades.
In 2007 we extended this classification to arbitrary integer symmetric matrices.

In this paper we turn our attention to symmetrizable matrices.
We classify the connected nonsymmetric but symmetrizable matrices which have entries in $\Z$ that are maximal with respect to having all their eigenvalues in $[-2,2]$. 
This includes a spectral characterisation of the affine and finite Dynkin diagrams that are not simply laced (much as the graph result gives a spectral characterisation of the simply laced ones).
\end{abstract}

\section{Introduction}

The simply laced finite and affine Dynkin diagrams $A_n$, $\widetilde A_n$ ($n\ge 1$), $D_n$, $\widetilde D_n$ ($n\ge 4$) and $E_n$, $\widetilde E_n$ ($n=6,7,8$) (see Figure \ref{F-ADE})  are very well known, and appear, mysteriously, in a wide variety of contexts where they are used to classify various kinds of mathematical objects.

\begin{figure}[hb]
\[
\begin{xy}
@={(0,0)="a", (6,0), (12,0)="b", (12,6), (12,12)="c", (18,0), (24,0)="d", (36,0)="e", (42,0), (48,0), (54,0)="f", (54,6)="g", (60,0), (66,0), (72,0)="h", (84,0)="i", (90,0), (96,0)="j", (96,6)="k", (102,0), (108,0), (114,0), (120,0), (126,0)="l", (18,18)="m", (24,18), (30,18), (36,24)="o", (42,18), (48,18), (54,18)="q", (78,18)="r", (80,24)="s", (84,18)="t", (90,18), (102,18), (108,18)="w", (112,24)="x", (114,18)="y"},
@@{*\frm<3pt>{*}},
"a";"d" **@{-}, "b";"c" **@{-},
"d" *\frm<6pt>{o},
(12,-5) *{E_6,\ \widetilde{E}_6},
"e";"h" **@{-}, "f";"g" **@{-}, "h" *\frm<6pt>{o},
(54,-5) *{E_7,\ \widetilde{E}_7},
"i";"l" **@{-}, "j";"k" **@{-}, "l" *\frm<6pt>{o},
(105,-5) *{E_8,\ \widetilde{E}_8},
"m";(32,18) **@{-}, (34,18);(38,18) **@{*}, 
(40,18);"q" **@{-}, "m";"o" **@{-}, "o";"q" **@{-},
"o" *\frm<6pt>{o}, (36,13) *{A_n,\ \widetilde{A}_n\ (n\ge 1)},
"r";(92,18) **@{-}, (94,18);(98,18) **@{*},
(100,18);"y" **@{-}, "s";"t" **@{-}, "w";"x" **@{-},
"x" *\frm<6pt>{o},
(96,13) *{D_n,\ \widetilde{D}_n\ (n\ge 4)}
\end{xy}
\]
\caption{The simply laced Dynkin diagrams. The circled vertices are for the affine diagrams only.}\label{F-ADE}
\end{figure}
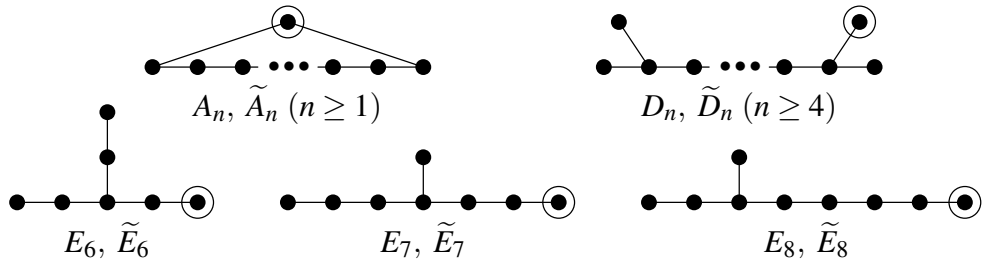

See for instance Arnold \cite{A} for a discussion of parts of the mystery, and Hazewinkel {\it et al} \cite{HHSV} for many of the contexts. 
One classification context not mentioned in \cite{HHSV}, however, is a simple spectral one:  the finite Dynkin diagrams  $A_n\,(n\ge 1)$, $D_n\,(n\ge 4)$ and $E_n\,(n=6,7,8)$ are the graphs with the property of being connected and having all their eigenvalues in the open interval $(-2,2)$, while the affine Dynkin diagrams $\widetilde A_n\,(n\ge 2)$, $\widetilde D_n\,(n\ge 4)$ and $\widetilde E_n\,(n=6,7,8)$ are the graphs maximal with respect to the property of being both connected and having all eigenvalues in the closed interval $[-2,2]$.
(These eigenvalues are defined as the eigenvalues of the adjacency matrix of the graph.) 
These results are credited to J.H. Smith \cite{Smi} in the graph-theoretic literature, but perhaps were known earlier by other specialists.
By Cauchy interlacing \cite{F} all induced subgraphs of such graphs  also have their eigenvalues in $[-2,2]$. 
This implies that the connected graphs with all eigenvalues in $[-2,2]$ are precisely the connected induced subgraphs of one of these affine Dynkin diagrams.

Since graphs can be considered, via their adjacency matrices, to be symmetric matrices all of whose entries are $0$ or $1$, with all diagonal elements $0$, it is natural to identify the two. 
In 2007 we extended the Smith results to arbitrary integer symmetric matrices \cite[Theorems 1, 2, 3]{MS2007}. 
Those matrices corresponded, with trivial exceptions (the $1\times 1$ matrices $\pm(2)$ and the $2\times 2$ adjacency matrix of $\pm\,\widetilde{A}_{1}$ (Figure \ref{F-AAA})) to {\it charged signed} graphs, i.e., to graphs with  edges weighted $\pm 1$ and `charges' of $\pm 1$ (or $0$) at vertices (see \cite[Theorems 1, 2, 3]{MS2007}).   
Thus all connected integer symmetric matrices maximal with respect to having all eigenvalues in $[-2,2]$, except $\pm(2)$ and $\pm\, \widetilde{A}_1$, have entries $0$ or $\pm 1$. 
Interestingly, these charged signed graphs actually have {\it all} their eigenvalues at $-2$ and $2$, and indeed the square of their adjacency matrix is always $4$ times the identity matrix. 
Smith's ADE result then follows  by finding the  connected induced subgraphs of these charged signed graphs that are maximal with respect to the property of having no charges or negative edges (allowing sign switchings: see Definition \ref{D:equivalence} below).

Recently, we have turned our attention to characterising spectrally the affine and finite Dynkin diagrams that are not simply laced---see Figures \ref{F-ADynk} and \ref{F-FDynk}. 
We shall show (Corollary \ref{C-1}) that the affine ones correspond to the connected nonsymmetric but symmetrizable matrices having entries in $\N0 = \{0,\ 1,\ 2,\ 3,\ \dots\}$ that are maximal with respect to the property of having all their eigenvalues in $[-2,2]$. 
In the language of \cite{MS2007} these would be called \textit{cyclotomic} symmetrizable integer matrices, in that their eigenvalues naturally map to roots of unity (if $\lambda = 2\cos\theta$ is an eigenvalue, then $\lambda \mapsto e^{\pm i\theta}$).
Note that the results show that for such a matrix $A=(a_{ij})$ the only pairs  $\{a_{ij},a_{ji}\}$ with $a_{ij}\ne a_{ji}$ that occur are $\{1,2\}$, $\{1,3\}$ or  $\{1,4\}$, with only one variety of such pairs occurring in each digraph.

For the finite Dynkin diagrams $B_n,C_n,F_4,G_2$ we show (Corollary \ref{C-5}) that they correspond to the non-symmetric symmetrizable integer matrices maximal with respect to having all their eigenvalues in $(-2,2)$.

We complete the spectral picture for nonsymmetric, symmetrizable connected integer matrices having all their eigenvalues in $[-2,2]$ in Theorem \ref{T-1} below. 
We also deduce a corresponding result (Theorem \ref{T-2}) for 
such matrices having all their eigenvalues in $(-2,2)$.
We emphasise that the novel feature of these results is that some integer matrix entries are allowed to be negative.

Our digraph naming follows \cite{Carter2005} for known digraphs. 
In particular, for the (tilded) affine Dynkin diagrams their subscript is  one fewer than the number of vertices. 
For all other digraphs, including newly defined ones, the subscript will equal the number of vertices.

\section{Understanding the diagrams}
To any square matrix $A = (a_{ij})$ with entries in $\Z$ we associate a digraph, also called $A$. 
(We shall refer to $A$ interchangeably as a digraph and as a matrix, sometimes within the same sentence.)
To the $i$th row of the matrix $A$ we associate a vertex $i$ of the digraph.
The diagonal entry $a_{ii}$ represents a \emph{charge} on the vertex $i$.
We shall be considering all possible integral charges, but the only ones we shall need to draw are $0$, $1$ and $-1$, for which we call the vertices neutral, positively charged and negatively charged respectively, and we draw them as\ \ $\begin{xy}
0 *\frm<3pt>{*}
\end{xy}$\ , $\begin{xy}
0 *=++{+}*\frm{o}
\end{xy}$ and $\begin{xy}
0 *=++{-}*\frm{o}
\end{xy}$ respectively.
The directed edge weights $a_{ij}$ (from $i$ to $j$) are arbitrary integers, but we shall only need to draw digraphs for which
\[
\{a_{ij},a_{ji}\} \in \bigl\{ \{0,0\},\ \{1,1\},\ \{-1,-1\},\ \{1,2\},\ 
\{1,3\},\ \{1,4\} \, \{-1, -2\}
\bigr\}\,.
\]
We represent these pairs of directed edges by drawing a single labelled edge (or no edge) as shown (here picturing all the vertices as being neutral).
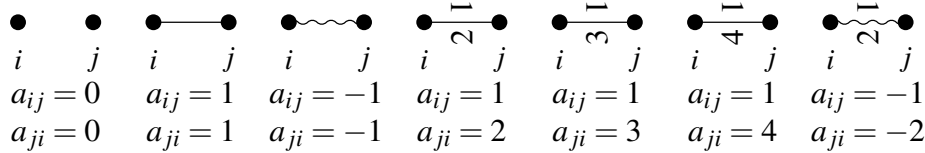
\[
\begin{xy}
@={(0,0), (10,0), (18,0)="a", (28,0)="b", (36,0)="c", (46,0)="d", (54,0)="e", (64,0)="f", (72,0)="g", (82,0)="h", (90,0)="i", (100,0)="j"},
@@{*\frm<3pt>{*}},
(0,-5) *\txt\small{$i$}, (10,-5) *\txt\small{$j$}, (5,-10) *{a_{ij}=0}, (5,-15) *{a_{ji}=0},
"a";"b" **@{-},
(18,-5) *\txt\small{$i$}, (28,-5) *\txt\small{$j$}, (23,-10) *{a_{ij}=1}, (23,-15) *{a_{ji}=1},
"c";"d" **@{~},
(36,-5) *\txt\small{$i$}, (46,-5) *\txt\small{$j$}, (41,-10) *{a_{ij}=-1}, (41,-15) *{a_{ji}=-1},
"e";"f" **@{-}, (59,2) *[@!-90]\txt\tiny{$1$}, (59,-2) *[@!90]\txt\tiny{$2$},
(54,-5) *\txt\small{$i$}, (64,-5) *\txt\small{$j$}, (59,-10) *{a_{ij}=1}, (59,-15) *{a_{ji}=2},
"g";"h" **@{-}, (77,2) *[@!-90]\txt\tiny{$1$}, (77,-2) *[@!90]\txt\tiny{$3$},
(72,-5) *\txt\small{$i$}, (82,-5) *\txt\small{$j$}, (77,-10) *{a_{ij}=1}, (77,-15) *{a_{ji}=3},
"i";"j" **@{-}, (95,2) *[@!-90]\txt\tiny{$1$}, (95,-2) *[@!90]\txt\tiny{$4$},
(90,-5) *\txt\small{$i$}, (100,-5) *\txt\small{$j$}, (95,-10) *{a_{ij}=1}, (95,-15) *{a_{ji}=4},
@i @={(108,0)="k", (118,0)="l"}, @@{*\frm<3pt>{*}},
"k";"l" **@{~}, (113,2) *[@!-90]\txt\tiny{$1$}, (113,-2) *[@!90]\txt\tiny{$2$},
(108,-5) *\txt\small{$i$}, (118,-5) *\txt\small{$j$}, (113,-10) *{a_{ij}=-1}, (113,-15) *{a_{ji}=-2},
\end{xy}
\]
In the asymmetric cases, the value $a_{ij}$ is written on the \emph{left} of the edge as we travel from $i$ to $j$ (and $a_{ji}$ is on the left as we travel from $j$ to $i$).

Given a digraph $A$, the matrix $A$ is determined only once the vertices have been given an ordering, but we regard all possible such matrices as equivalent (Definition \ref{D:equivalence}).
For example, the digraph $B_2^{\pm}$ of Figure \ref{F-A12} corresponds either to $\begin{pmatrix}
1 & 1 \\ 2 & -1
\end{pmatrix}$ or $\begin{pmatrix}
-1 & 2 \\ 1 & 1
\end{pmatrix}$, but not $\begin{pmatrix}
1 & 2 \\ 1 & -1
\end{pmatrix}$ .
Naturally for a digraph $A$ we define the digraph $A^\T$ to be that which corresponds to the matrix $A^\T$ (the transpose of $A$).

The diagrams of Figures \ref{F-A13G}--\ref{F-WrG} correspond to symmetric matrices with entries in $\sqrt{\N0}=\{0,\ 1,\ -1,\ \sqrt{2},\ -\sqrt{2}, \sqrt{3},\ -\sqrt{3},\ \dots\}$.
In this symmetric case we draw the edges as shown.
\[
\begin{xy}
@={(0,0), (10,0), (20,0)="a", (30,0)="b", (40,0)="c", (50,0)="d", (60,0)="e", (80,0)="f", (90,0)="g", (110,0)="h"},
@@{*\frm<3pt>{*}},
(0,-5) *\txt\small{$i$}, (10,-5) *\txt\small{$j$}, (5,-10) *{a_{ij}=0},
"a";"b" **@{-},
(20,-5) *\txt\small{$i$}, (30,-5) *\txt\small{$j$}, (25,-10) *{a_{ij}=1},
"c";"d" **@{~},
(40,-5) *\txt\small{$i$}, (50,-5) *\txt\small{$j$}, (45,-10) *{a_{ij}=-1},
"e";(66,0) **@{-}, (74,0);"f" **@{-}, (70,0) *\txt\tiny{$\sqrt{a}$}, 
(60,-5) *\txt\small{$i$}, (80,-5) *\txt\small{$j$}, (70,-10) *{a_{ij}=\sqrt{a}},
"g";(96,0) **@{~}, (104,0);"h" **@{~}, (100,0) *\txt\tiny{$\sqrt{a}$}, 
(90,-5) *\txt\small{$i$}, (110,-5) *\txt\small{$j$}, (100,-10) *{a_{ij}=-\sqrt{a}}
\end{xy}
\]

\begin{figure}[hb]
\[
\begin{xy}
@={(0,5)="a", (0,25)="b", (20,5)="c", (20,25)="d", (40,5)="e", (40,25)="f", (60,5)="g", (60,25)="h", (80,0)="i", (80,20)="j", (90,5)="k", (90,25)="l", (100,0)="m", (100,20)="n", (110,5)="o", (110,25)="p"},
@@{*\frm<3pt>{*}},
(0,0) *{\widetilde{A}_1},
"a";"b" **@{-},
(-2,15) *\txt\small{2}, (2,15) *[@!180]\txt\small{2},
(20,0) *{\widetilde{A}_1'},
"c";"d" **@{-},
(18,15) *\txt\small{4}, (21.5,15) *[@!180]\txt\small{1},
(50,15) *{O_4'},
"e";"f" **@{~}, "e";"g" **@{-}, "f";"h" **@{-}, "g";"h" **@{-},
(50,3) *[@!90]\txt\small{3}, (50,7) *[@!-90]\txt\small{1},
(50,23) *[@!90]\txt\small{3}, (50,27) *[@!-90]\txt\small{1},
(110,0) *{S_8^-},
"i";"j" **@{-}, "i";"k" **@{-}, "i";"m" **@{-},
"j";"l" **@{~}, "j";"n" **@{-}, "k";"l" **@{-}, "k";"o" **@{~}, 
"l";"p" **@{-}, "m";"n" **@{~}, "m";"o" **@{-},
"n";"p" **@{-}, "o";"p" **@{-},
(78,10) *\txt\small{2}, (81.5,10) *[flip]\txt\small{1},
(88,15) *\txt\small{2}, (91.5,15) *[flip]\txt\small{1},
(98,10) *\txt\small{2}, (102,10) *[flip]\txt\small{1},
(108,15) *\txt\small{2}, (111.5,15) *[flip]\txt\small{1},
\end{xy}
\]
\caption{ $\widetilde{A}_{1}$, $\widetilde{A}_{1}'$, $O_4'$ and ${S}_{8}^{-}$.}\label{F-AAA}
\end{figure}

\begin{figure}[ht]
\[
\begin{xy}
@={(0,10)="a", (10,0)="b", (10,20)="c", (20,0)="d", (20,20)="e", (30,0)="f", (30,20)="g", (50,0)="l", (50,20)="m", (60,0)="n", (60,20)="o", (70,0)="p", (70,20)="q", (80,10)="r"},
@@{*\frm<3pt>{*}},
(37,0)="h", (37,20)="i",
(43,0)="j", (43,20)="k",
(33,0)="f2", (33,20)="g2", (47,0)="l2", (47,20)="m2",
"a";"b" **@{-}, "a";"c" **@{-},
"b";"f2" **@{~}, "b";"e" **@{~},
"c";"d" **@{-}, "c";"g2" **@{-},
"d";"g" **@{~}, "e";"f" **@{-},
"h";"j" **@{*}, "i";"k" **@{*},
"m2";"q" **@{-}, "l2";"p" **@{~},
"l";"o" **@{~}, "m";"n" **@{-},
"n";"q" **@{~}, "o";"p" **@{-},
"p";"r" **@{~}, "q";"r" **@{-},
(4,4) *[@!45]\txt\small{1}, (6.5,6.5) *[@!-135]\txt\small{2},
(3.5,16.5) *[@!-45]\txt\small{2}, (6,13.5) *[@!135]\txt\small{1},
(73.5,13) *[@!45]\txt\small{$2'$}, (77,16.5) *[@!-135]\txt\small{$1'$},
(73,6) *[@!-45]\txt\small{$1'$}, (76.5,3) *[@!135]\txt\small{$2'$},
\end{xy}
\]
\caption{ ${L_n}$ and ${L_n'}$, for $n=2r+2\ge 4$ and even, where $r$ is the number of vertices on the top row. The two weight pairs $2',1'$ are $2,1$ for $L_n$, but are swapped to $1,2$ for $L_n'$. Note that $L_n'$ is equivalent to its transpose, but if $n\ge 6$ then $L_n$ is not (Lemma \ref{L-equiv} and Corollary \ref{C-inequiv}).}\label{F-Vr}
\end{figure}
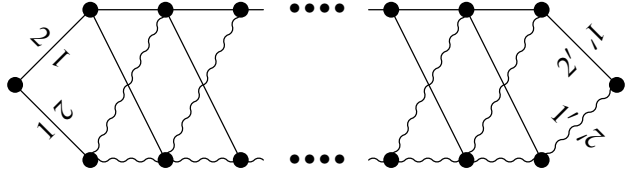

\begin{figure}[ht]
\[
\begin{xy}
@={(40,10)="g", (50,0)="h", (50,20)="i", (60,0)="j", (60,20)="k", (70,0)="l", (70,20)="m", (90,0)="n", (90,20)="o", (100,0)="p", (100,20)="q"},
@@{*\frm<3pt>{*}},
(0,0) *=++{+}*\frm{o}="a", (0,20) *=++{-}*\frm{o}="b",
(10,0) *=++{-}*\frm{o}="c", (10,20) *=++{+}*\frm{o}="d",
(30,0) *=++{+}*\frm{o}="e", (30,20) *=++{-}*\frm{o}="f",
(110,0) *=++{+}*\frm{o}="r", (110,20) *=++{+}*\frm{o}="s",
(73,0)="t", (77,0)="u", (83,0)="v", (87,0)="w",
(73,20)="x", (77,20)="y", (83,20)="z", (87,20)="aa",
(-8,0) *{{A}_2^\pm}, (20,10) *{O_4^\pm},
(80,10) *{L_n^+},
"a";"b" **@{-},
(-2,10) *\txt\small{3}, (1.5,10) *[flip]\txt\small{1},
"c";"d" **@{-}, "c";"e" **@{-}, "d";"f" **@{-}, "e";"f" **@{~},
(20,-2) *[@!90]\txt\small{1}, (20,2) *[@!-90]\txt\small{2},
(20,18) *[@!90]\txt\small{1}, (20,22) *[@!-90]\txt\small{2},
"g";"h" **@{-}, "g";"i" **@{-},
"h";"t" **@{~}, "h";"k" **@{~},
"i";"j" **@{-}, "i";"x" **@{-},
"j";"m" **@{~}, "k";"l" **@{-}, 
"u";"v" **@{*}, "y";"z" **@{*},
"w";"r" **@{~}, "aa";"s" **@{-},
"n";"q" **@{~}, "o";"p" **@{-},
"p";"s" **@{~}, "q";"r" **@{-}, "r";"s" **@{~},
(43.5,3.5) *[@!45]\txt\small{2}, (46,6) *[@!-135]\txt\small{1},
(44,16) *[@!-45]\txt\small{1}, (46,13.5) *[@!135]\txt\small{2}
\end{xy}
\]
\caption{$A_2^{\pm}$, $O_4^{\pm}$  and $L_n^{+}$, for $n=2r+1\ge 3$, where $r$ is the number of vertices on its top row, including the charged vertex.}\label{F-Wr}
\end{figure}
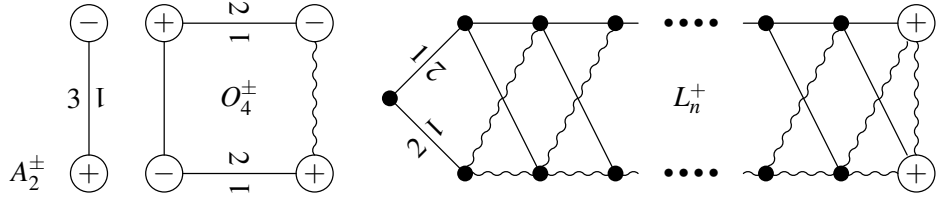

\begin{figure}[ht]
\[
\begin{xy}
@={(0,0)="a", (8,0)="b", (8,12)="c", (8,20)="d", (16,0)="e", (16,16)="f", (24,16)="g", (32,0)="h", (32,16)="i", (40,0)="j", (40,16)="k", (56,0)="l", (56,16)="m", (64,0)="n", (64,16)="o", (72,0)="p", (80,0)="q", (80,16)="r", (88,0)="s", (88,16)="t", (96,16)="u", (104,0)="v", (112,0)="w", (112,16)="x", (120,0)="y", (120,16)="z"},
@@{*\frm<3pt>{*}},
(18,0)="aa", (21,0)="bb", (27,0)="cc", (30,0)="dd",
(42,16)="ee", (45,16)="ff", (51,16)="gg", (54,16)="hh",
(98,16)="ii", (101,16)="jj", (107,16)="kk", (110,16)="ll",
"c";"f" **@{-}, "d";"f" **@{-}, "f";"ee" **@{-},
"ff";"gg" **@{*}, "hh";"o" **@{-},
(60,14) *[@!90]\txt\small{1}, (60,18) *[@!-90]\txt\small{2},
(3,16) *{\widetilde{B}_n},
"r";"ii" **@{-}, "jj";"kk" **@{*}, "ll";"z" **@{-},
(84,14) *[@!90]\txt\small{2}, (84,18) *[@!-90]\txt\small{1},
(116,14) *[@!90]\txt\small{1}, (116,18) *[@!-90]\txt\small{2},
(75,16) *{\widetilde{C}_n},
(-5,0) *{\widetilde{C}_n'},
(51,0) *{\widetilde{F}_4}, (99,0) *{\widetilde{G}_2},
"a";"aa" **@{-}, "bb";"cc" **@{*}, "dd";"j" **@{-},
(4,-2) *[@!90]\txt\small{2}, (4,2) *[@!-90]\txt\small{1},
(36,-2) *[@!90]\txt\small{2}, (36,2) *[@!-90]\txt\small{1},
"l";"s" **@{-}, 
(76,-2) *[@!90]\txt\small{2}, (76,2) *[@!-90]\txt\small{1},
"v";"y" **@{-},
(116,-2) *[@!90]\txt\small{3}, (116,2) *[@!-90]\txt\small{1},
\end{xy}
\]
\caption{The symmetrizable but nonsymmetric affine Dynkin diagrams $\widetilde{B}_n\,(n\ge 3)$, $\widetilde{C}_n\,(n\ge 2)$, $\widetilde{C}_n'\,(n\ge 2)$, $\widetilde{F}_4$, $\widetilde{G}_2$.  
Note that only $\widetilde{C}_n'$ is equivalent to its transpose (Lemma \ref{L-equiv} and Corollary \ref{C-inequiv}).} \label{F-ADynk}
\end{figure}
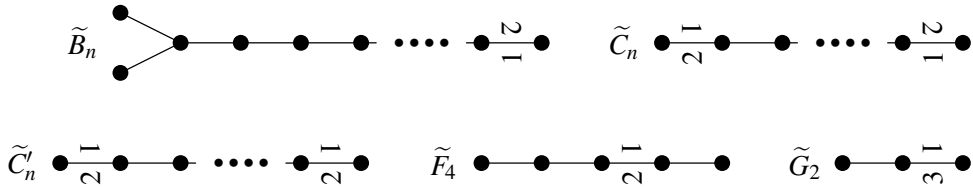

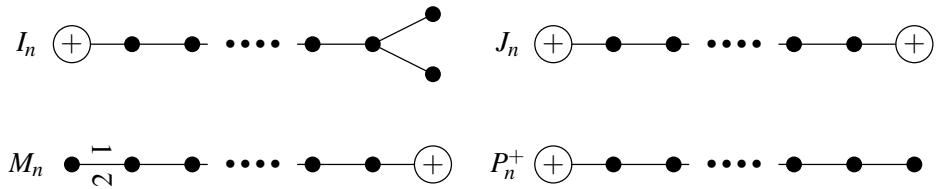
\begin{figure}[b]
\[
\begin{xy}
@={(0,0)="a", (8,0), (8,16), (16,0), (16,16), (32,0), (32,16), (40,0), (40,16)="k", (48,12)="m", (48,20)="n", (72,0), (72,16), (80,0), (80,16), (96,0), (96,16), (104,0), (104,16), (112,0)="y"},
@@{*\frm<3pt>{*}},
(0,16) *=++{+}*\frm{o}="b", (48,0) *=++{+}*\frm{o}="l",
(64,0) *=++{+}*\frm{o}="o", (64,16) *=++{+}*\frm{o}="p",
(112,16) *=++{+}*\frm{o}="z",
(18,0)="c", (18,16)="d", (21,0)="e", (21,16)="f",
(27,0)="g", (27,16)="h", (30,0)="i", (30,16)="j",
(82,0)="q", (82,16)="r", (85,0)="s", (85,16)="t",
(91,0)="u", (91,16)="v", (94,0)="w", (94,16)="x",
(-6,16) *{I_n}, "b";"d" **@{-}, "f";"h" **@{*},
"j";"k" **@{-}, "k";"m" **@{-}, "k";"n" **@{-},
(58,16) *{J_n}, "p";"r" **@{-}, "t";"v" **@{*}, "x";"z" **@{-},
(-6,0) *{M_n}, "a";"c" **@{-}, "e";"g" **@{*}, "i";"l" **@{-},
(4,-2) *[@!90]\txt\small{2}, (4,2) *[@!-90]\txt\small{1},
(58,0) *{P_n^+}, "o";"q" **@{-},
"s";"u" **@{*}, "w";"y" **@{-}
\end{xy}
\]
\caption{ The $n$-vertex digraphs $I_n\,(n\ge 3)$, and $J_n$, $M_{n}$ and $P_n^{+}$ for $n\ge 2$. 
Only $M_n$ is nonsymmetric.}\label{F-Mr}
\end{figure}

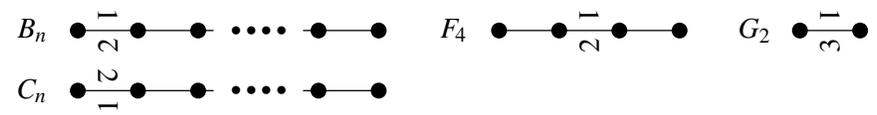
\begin{figure}[b]
\[
\begin{xy}
@={(0,0)="a", (8,0), (16,0), (32,0), (40,0)="f", (56,0)="g", (64,0), (72,0), (80,0)="h", (96,0)="i", (104,0)="j"},
@@{*\frm<3pt>{*}},
(18,0)="b", (21,0)="c", (27,0)="d", (30,0)="e", 
"a";"b" **@{-}, "c";"d" **@{*}, "e";"f" **@{-},
(-6,0) *{B_n}, (4,-2) *[@!90]\txt\small{2},
(4,2) *[@!-90]\txt\small{1},
(50,0) *{F_4}, (68,-2) *[@!90]\txt\small{2},
(68,2) *[@!-90]\txt\small{1},
(90,0) *{G_2}, (100,-2) *[@!90]\txt\small{3},
(100,2) *[@!-90]\txt\small{1},
"g";"h" **@{-}, "i";"j" **@{-},
(0,-8)*\frm<3pt>{*}="A", (8,-8)*\frm<3pt>{*}="B", (16,-8)*\frm<3pt>{*}="C", (21,-8)="D", (27,-8)="E",
(30,-8)="F", (32,-8)*\frm<3pt>{*}="G", (40,-8)*\frm<3pt>{*}="H",
"A";(18,-8) **@{-}, "D";"E" **@{*}, "F";"H" **@{-},
(-6,-8) *{C_n}, (4,-10) *[@!90]\txt\small{1},
(4,-6) *[@!-90]\txt\small{2}
\end{xy}
\]
\caption{The symmetrizable but nonsymmetric finite Dynkin diagrams $B_n\,(n\ge 2)$, $F_4$ and $G_2$; $C_n\,(n\ge 3)=B_n^\T$. Note that $B_2$, $F_4$ and $G_2$ are equivalent to their transposes, while $B_n\,(n\ge 3)$ is not (Lemma \ref{L-equiv} and Corollary \ref{C-inequiv}). }\label{F-FDynk}
\end{figure}

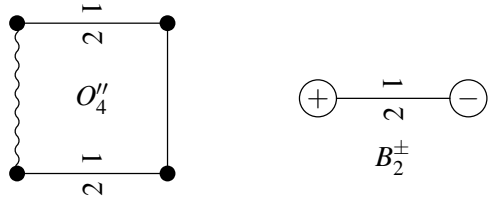
\begin{figure}[ht]
\[
\begin{xy}
@={(0,0)="a", (0,20)="b", (20,0)="c", (20,20)="d"},
@@{*\frm<3pt>{*}},
(40,10) *=++{+}*\frm{o}="e", (60,10) *=++{-}*\frm{o}="f",
(10,10) *{O_4''},
"a";"b" **@{~}, "a";"c" **@{-}, "b";"d" **@{-}, "c";"d" **@{-},
(10,-2) *[@!90]\txt\small{2}, (10,2) *[@!-90]\txt\small{1},
(10,18) *[@!90]\txt\small{2}, (10,22) *[@!-90]\txt\small{1},
(50,2) *{B_2^\pm},
"e";"f" **@{-},
(50,8) *[@!90]\txt\small{2}, (50,12) *[@!-90]\txt\small{1}
\end{xy}
\]
\caption{ The digraphs   $O_4''$ and   ${B}_{2}^{\pm}$. }\label{F-A12}
\end{figure}

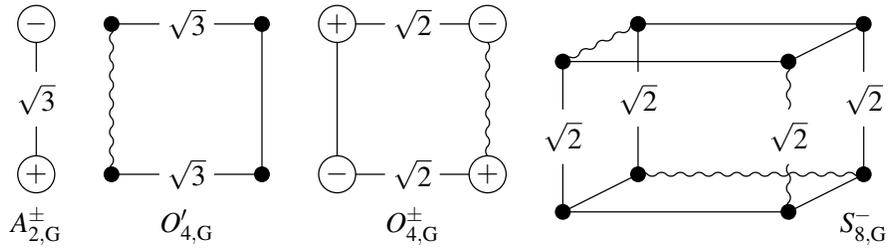
\begin{figure}[ht]
\[
\begin{xy}
(0,5) *=++{+}*\frm{o}="a", (0,15) *=+++{\sqrt{3}}*\frm{}="b",
(0,25) *=++{-}*\frm{o}="c",
"a";"b" **@{-}, "b";"c" **@{-}, (0,-2) *{A_{2\G}^\pm},
@={(10,5)="a", (10,25)="b", (30,5)="e", (30,25)="f"},
@@{*\frm<3pt>{*}},
"a";"b" **@{~}, "e";"f" **@{-},
(20,5) *=+++{\sqrt{3}}="c", (20,25) *=+++{\sqrt{3}}="d",
"a";"c" **@{-}, "c";"e" **@{-}, "b";"d" **@{-}, "d";"f" **@{-},
(20,-2) *{O_{4\G}'},
(40,5) *=++{-}*\frm{o}="a", (40,25) *=++{+}*\frm{o}="b",
(60,5) *=++{+}*\frm{o}="e", (60,25) *=++{-}*\frm{o}="f",
(50,5) *=+++{\sqrt{2}}="c", (50,25) *=+++{\sqrt{2}}="d",
(50,-2) *{O_{4\G}^\pm},
"a";"c" **@{-}, "c";"e" **@{-}, "b";"d" **@{-}, "d";"f" **@{-},
"a";"b" **@{-}, "e";"f" **@{~},
@i, @={(70,0)="a", (70,20)="b", (80,5)="c", (80,25)="d", (100,0)="e", (100,20)="f", (110,5)="g", (110,25)="h"},
@@{*\frm<3pt>{*}},
(70,10) *=+++{\sqrt{2}}="i", (80,15) *=+++{\sqrt{2}}="j",
(100,10) *=+++{\sqrt{2}}="k", (110,15) *=+++{\sqrt{2}}="l",
"a";"i" **@{-}, "a";"c" **@{-}, "a";"e" **@{-},
"b";"i" **@{-}, "b";"d" **@{~}, "b";"f" **@{-},
"c";"j" **@{-}, "c";"g" **@{~}, "d";"j" **@{-},
"d";"h" **@{-}, "e";"k" **@{~}, "e";"g" **@{-},
"f";"k" **@{~}, "f";"h" **@{-}, "g";"l" **@{-},
"h";"l" **@{-},
(110,-2) *{S_{8\G}^-}
\end{xy}
\]
\caption{ The Greaves graphs ${A}_{2\G}^{\pm}$,  $O_{4\G}'$,  $O_{4\G}^{\pm}$ and $S_{8\G}^{-}$. Since they are symmetric, a single weight is given for each edge.}\label{F-A13G}
    
\end{figure}

\begin{figure} [ht]
\[
\begin{xy}
@={(-10,10)="a", (10,0)="d", (10,20)="e", (20,0)="f", (20,20)="g", (30,0)="h", (30,20)="i", (50,0)="r", (50,20)="s", (60,0)="t", (60,20)="u", (70,0)="v", (70,20)="w", (90,10)="z"},
(0,5) *=+++{\sqrt{2}}*\frm{}="b", (0,15)*=+++{\sqrt{2}}*\frm{}="c", (34,0)="j", (34,20)="k", (36,0)="l", (36,20)="m", (44,0)="n", (44,20)="o", (46,0)="p", (46,20)="q", (80,5) *=+++{\sqrt{2}}*\frm{}="x", (80,15) *=+++{\sqrt{2}}*\frm{}="y",
@@{*\frm<3pt>{*}},
"a";"b" **@{-}, "a";"c" **@{-}, "b";"d" **@{-},
"c";"e" **@{-}, "d";"j" **@{~}, "d";"g" **@{~},
"e";"f" **@{-}, "e";"k" **@{-}, "f";"i" **@{~},
"g";"h" **@{-}, "l";"n" **@{*}, "m";"o" **@{*},
"p";"v" **@{~}, "q";"w" **@{-}, "r";"u" **@{~},
"s";"t" **@{-}, "t";"w" **@{~}, "u";"v" **@{-},
"v";"x" **@{~}, "w";"y" **@{-}, "x";"z" **@{~}, "y";"z" **@{-}
\end{xy}
\]
\caption{ ${L}_{n\G}$ for $n=2r+2\ge 4$, where $r$ is the number of vertices on the top row.}\label{F-VrG}
\end{figure}
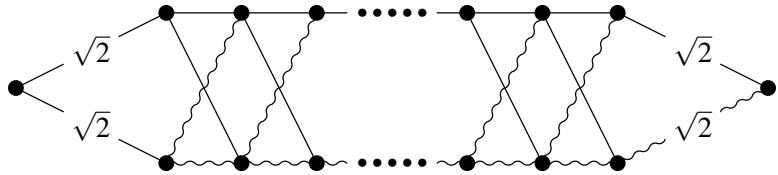

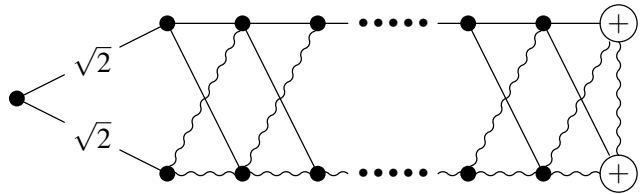
\begin{figure}[ht] 
\[
\begin{xy}
@={(-10,10)="a", (10,0)="d", (10,20)="e", (20,0)="f", (20,20)="g", (30,0)="h", (30,20)="i", (50,0)="r", (50,20)="s", (60,0)="t", (60,20)="u"},
(70,0) *=++{+}*\frm{o}="v", (70,20) *=++{+}*\frm{o}="w",
(0,5) *=+++{\sqrt{2}}*\frm{}="b", (0,15)*=+++{\sqrt{2}}*\frm{}="c", (34,0)="j", (34,20)="k", (36,0)="l", (36,20)="m", (44,0)="n", (44,20)="o", (46,0)="p", (46,20)="q", 
@@{*\frm<3pt>{*}},
"a";"b" **@{-}, "a";"c" **@{-}, "b";"d" **@{-},
"c";"e" **@{-}, "d";"j" **@{~}, "d";"g" **@{~},
"e";"f" **@{-}, "e";"k" **@{-}, "f";"i" **@{~},
"g";"h" **@{-}, "l";"n" **@{*}, "m";"o" **@{*},
"p";"v" **@{~}, "q";"w" **@{-}, "r";"u" **@{~},
"s";"t" **@{-}, "t";"w" **@{~}, "u";"v" **@{-},
"v";"w" **@{~}
\end{xy}
\]
\caption{ ${L}_{n\G}^{+}$ for $n=2r+1\ge 3$, where $r$ is the number of vertices on the top row.}\label{F-WrG}
\end{figure}

\clearpage

\section{Results}

Full definitions will follow (maximal, connected, symmetrizable, equivalent, \dots), but for convenience we present here in one place the main results of our paper.

Our first main result is the following.
\begin{theorem} \label{T-1}
Let $A$ be a connected symmetrizable integer matrix, maximal with respect to having all its eigenvalues in the interval $[-2,2]$. 
If $A$ is not symmetric, then $A$ is equivalent to one of the following digraphs (see Figures \textup{\ref{F-AAA}}, \textup{\ref{F-Vr}}, \textup{\ref{F-Wr}}).

Not charged:
 $\widetilde{A}_{1}'$,
 $O_4'$,
 ${S}_8^{-}$,
 ${L}_{n}$   and $L_n'\,(n\ge 4$ and even),
 $L_n^\T(n\ge 6,even)$.

Charged:
 ${O}_{4}^{\pm }$, 
 ${{A}}_{2}^{\pm }$, 
 ${L}_{n}^{+}$, $(L_n^+)^\T\,(n\ge 3$ and odd).

All these digraphs (matrices) $A$ have $A^2=4I$.
Furthermore, every connected symmetrizable nonsymmetric matrix having all its eigenvalues in the interval $[-2,2]$ is contained in a maximal one.
\end{theorem}

By combining our earlier work \cite[Theorems 1, 2, 3]{MS2007} with these results, and separating the charged and noncharged cases, we obtain the following.
 
\begin{corollary} \label{C-1} 
Let $A$ be an $n\times n$ connected, sign symmetric matrix with \textbf{nonnegative} integer entries and maximal with respect to having all its eigenvalues in $[-2,2]$. 
Suppose first that $A$ is uncharged. 
\begin{itemize} 
\item If $A$ is symmetric then it is equivalent to the $1\times 1$ matrix $(2)$, the $2\times 2$ matrix $\widetilde A_1$ (Figure \textup{\ref{F-AAA}}) or to one of the simply laced Dynkin diagrams  $\widetilde A_n\,(n\ge 2)$, $\widetilde D_n\,(n\ge 4)$ , $\widetilde E_6$,  $\widetilde E_7$ or $\widetilde E_8$ (Figure \ref{F-ADE}). 
            
\item If $A$ is nonsymmetric  then it is equivalent to one of the non-simply laced Dynkin diagrams $\widetilde A_{1}'$, $\widetilde{B}_n$, $\widetilde{B}_n^{\ \T},\,(n\ge 2)$, $\widetilde{C}_n\,(n\ge 2)$, $\widetilde{C}_n^{\ \T}\,(n\ge 2)$, $\widetilde{C}_n'\,(n\ge 3)$, $\widetilde{G}_{2}$, $\widetilde{G}_{2}^{\,\T}$,  $\widetilde{F}_{4}$ or $\widetilde{F}_{4}^{\,\T}$. 
See Figures \textup{\ref{F-AAA}} and \textup{\ref{F-ADynk}}. 
\end{itemize}
Now suppose that $A$ is charged.
\begin{itemize} 
\item If $A$ is symmetric then it is equivalent to one of $I_n\,(n\ge 3)$ or $J_n\,(n\ge 2)$ from Figure \textup{\ref{F-Mr}}.

\item If $A$ is nonsymmetric  then it equivalent to the digraph $M_n$ of Figure \textup{\ref{F-Mr}} for some $n\ge 2$, or to its transpose $M_n^{\T}$.
\end{itemize}

\end{corollary}

A vital ingredient in the proofs is the generalisation by Greaves \cite{G} of the classification of integer symmetric matrices of spectral radius at most $2$ \cite[Theorems 1, 2, 3]{MS2007} to symmetric matrices whose entries are algebraic integers lying in the compositum of all real quadratic fields.

Corollary \ref{C-1} implies the following.

\begin{corollary}\label{C-3} The symmetrizable nonsymmetric affine Dynkin diagrams (all those in Figure \textup{\ref{F-ADynk}}, and $\widetilde{A}_1'$ in Figure \textup{\ref{F-AAA}}), along with their transposes, are the connected subgraphs of the digraphs of Theorem \textup{\ref{T-1}} that are maximal with respect to the property that some sign switching of the subgraph has all eigenvalues in $[-2,2]$ and no charges or negative edges.
\end{corollary}

Now we consider the open interval $(-2,2)$, this time with a corollary relating to the finite Dynkin diagrams.

\begin{theorem}\label{T-2}
Let $A$ be a connected symmetrizable integer matrix having all its eigenvalues in the open interval $(-2,2)$.
If $A$ is not symmetric, then $A$ is an induced subgraph of a connected symmetrizable integer  matrix that is maximal with respect to the property of having all its eigenvalues in $(-2,2)$.
The maximal such $A$ are equivalent to one of the following graphs (see Figures \ref{F-FDynk} and \ref{F-A12}): $B_n$ ($n\ge 2$), $C_n := B_n^\T$ ($n\ge 3$), $F_4$, $G_2$, $O_4''$, or $B_2^\pm$.
\end{theorem}
Note the contrast to the symmetric case: not every connected \emph{symmetric} matrix having all its eigenvalues in the open interval $(-2,2)$ is an induced subgraph of a maximal one (see \cite[Theorem 6]{MS2007}).
As with Theorem \ref{T-1}, we can combine Theorem \ref{T-2} with results from \cite{MS2007} to obtain the following.

\begin{corollary}\label{C-5}
Let $A$ be a connected symmetrizable matrix with \textbf{nonnegative} integer entries having all eigenvalues in $(-2,2)$.

Suppose first that $A$ is uncharged.
\begin{itemize}
    \item If $A$ is symmetric then $A$ is equivalent to one of $A_n$ ($n\ge 1$), $D_n$ ($n\ge 4$), $E_6$, $E_7$, or $E_8$.
    \item If $A$ is symmetrizable but nonsymmetric, then $A$ is an induced subgraph of a matrix that is maximal with respect to being connected, symmetrizable, and having all its eigenvalues in $(-2,2)$.
    The maximal such matrices are those equivalent to one of $B_n$, $C_n = B_n^\T$, $F_4$, or $G_2$ (see Figure \ref{F-FDynk}).
\end{itemize}

Now suppose that $A$ is charged.
\begin{itemize}
    \item If $A$ is symmetric then it is equivalent to some $P_n^+$ ($n\ge 1$) (see Figure \ref{F-Mr}).
    \item There are no such $A$ that are nonsymmetric and symmetrizable.
\end{itemize}
\end{corollary}
\section{Definitions and preparatory lemmas}

\begin{definition}\label{D:symmetrizable}

An $n\times n$ real matrix $B$ is said to be \textbf{symmetrizable} if there is a real diagonal matrix $D = \diag(d_1,\dots,d_n)$ with each $d_i>0$ such that \begin{equation}\label{E:symmetrizable}
S=D^{-1}BD
\end{equation}
is symmetric.
We then call $S$ the \textbf{symmetrization} of $B$.

\end{definition}

Thus a symmetrizable matrix can be transformed to a symmetric matrix by a diagonal change of basis (and with positive scaling of each basis vector).
In particular, all the eigenvalues of a symmetrizable matrix are real, since they equal those of the symmetrization.
If $B$ is symmetrizable, then we shall see that its symmetrization $S$ is unique (although there will be choice for $D$ in \eqref{E:symmetrizable}).
There are at least two alternative equivalent definitions in the literature \cite{Carter2005,K}, but Definition \ref{D:symmetrizable} is possibly the most intuitive.

\begin{lemma}\label{L:transpose}
If $B$ is symmetrizable, then so is $B^\T$, and the symmetrizations of $B$ and $B^\T$ are the same.
\end{lemma}

\begin{proof}
Given \eqref{E:symmetrizable}, we transpose to get $S = D^\T B^\T(D^\T)^{-1}$.
\end{proof}

If $B=(b_{ij})$ is an $n\times n$ matrix, $D = \diag(d_1,\dots,d_n)$ with each $d_i>0$ and $S = D^{-1}BD$ is symmetric, then $d_i^{-1}b_{ij}d_j = d_j^{-1}b_{ji}d_i$ for $1\le i\le n$, $1\le j\le n$.
Hence
\begin{equation}\label{E:bd2}
    b_{ij}d_j^2 = b_{ji} d_i^2 \qquad (1\le i\le n,\ 1\le j\le n)\,.
\end{equation}

We call \eqref{E:bd2} the \textbf{balancing condition}.

\begin{definition}\label{D:signsymmetry}
A real $n\times n$ matrix $B=(b_{ij})$ is called \textbf{sign symmetric} if
\begin{equation}\label{E:signsymmetry}
    \sgn(b_{ij})=\sgn(b_{ji})
\end{equation}
holds for all $i$, $j\in\{1,\dots,n\}$.
(Here $\sgn(x)$ is the sign of $x$, either $1$, $-1$, or $0$.)
\end{definition}

\begin{lemma}\label{L:signsymmetry}
Any symmetrizable matrix is sign symmetric.
\end{lemma}

\begin{proof}
This is an immediate consequence of \eqref{E:bd2}.
\end{proof}

\begin{lemma}
If $B$ is a symmetrizable integer matrix, then its symmetrization has all entries in $\sqrt{\N0}=\{0,\ 1,\ -1,\ \sqrt{2},\ -\sqrt{2}, \sqrt{3},\ -\sqrt{3}, \dots \}$.
\end{lemma}

\begin{proof}
Suppose that $B=(b_{ij})$ is symmetrizable, and that positive diagonal $D = \diag(d_1,\dots,d_n)$ and symmetric $S=(s_{ij})$ satisfy $S=D^{-1}BD$.
We have
\begin{equation}\label{E:sij2}
s_{ij}^2 = s_{ij}s_{ji} = (d_{i}^{-1}b_{ij}d_j)(d_j^{-1}b_{ji}d_i) =b_{ij}b_{ji}\,,
\end{equation}
and by sign symmetry of $B$ this is in $\N0$.
\end{proof}

Regardless of whether or not the $b_{ij}$ are integers, combining \eqref{E:sij2} with
\begin{equation}\label{E:sgnsij}
\sgn(s_{ij}) = \sgn(d_i^{-1}b_{ij}d_j)=\sgn(b_{ij})
\end{equation}
we see that the symmetrization $S$ is uniquely determined by $B$.

A \textbf{permutation matrix} is a square matrix that has a single entry equal to $1$ in each row and column, and all other entries are zero.
A \textbf{signed permutation matrix} is like a permutation matrix but with the nonzero entries allowed to be either $1$ or $-1$.
(These matrices are the elements of the orthogonal group $O_n(\Z)$.)

\begin{definition}\label{D:equivalence}
Two square matrices $A$ and $B$ are called \textbf{equivalent} if there is a signed permutation matrix $P$ such that $P^{-1}AP = P^\T A P = \pm B$.

If the signed permutation matrix $P$ is diagonal, then the transformation $A\mapsto P^\T AP$ is called a \textbf{sign switching}.
\end{definition}

\begin{lemma}\label{L:equivalent}
If an integer matrix $B$ is symmetrizable, then so is any matrix equivalent to $B$.
\end{lemma}

\begin{proof}
Suppose that $D^{-1}BD=S$, where $D$ is diagonal with positive diagonal entries, and let $P$ be any signed permutation matrix of the same size.
Then 
we compute that 
$P^{-1}DP=P^\T DP$ is also diagonal with positive diagonal entries.
Since $(P^{-1}D^{-1}P)(P^{-1}BP)(P^{-1}DP)=P^\T SP$ is symmetric, we see that $P^{-1}BP$ is symmetrizable.
Moreover $D^{-1}(-B)D=-S$ is symmetric, so that the negative of a symmetrizable matrix is symmetrizable.
Hence any matrix equivalent to $B$ is symmetrizable.
\end{proof}

\begin{definition}\label{D:graphs}
Given a digraph $A$, an \textbf{induced subgraph} is a digraph formed from some subset of the vertices of $A$, with all charges and edge weights inherited from those in $A$.
In terms of matrices, an induced subgraph is produced by deleting some subset of the rows and deleting the same subset of the columns.

A digraph $A=(a_{ij})$ is \textbf{connected} if for any pair of vertices $x$ and $y$, there is a sequence of vertices
\[
v_1 = x,\ v_2,\ \dots,\ v_r = y
\]
such that for $i = 1$, \dots, $r-1$, there holds $a_{v_i v_{i+1}} \ne 0$.
(This corresponds to the usual definition of being \emph{strongly} connected, but for symmetrizable matrices the two properties coincide, since such matrices have sign symmetry.)

If $P$ is a property that a digraph might or might not have, then we say that $A$ is \textbf{maximal} with respect to that property if:
\begin{itemize}
    \item $A$ is connected and has property $P$;
    \item if $A$ is an induced subgraph of some strictly larger connected digraph $B$, then $B$ does not have property $P$.
\end{itemize}

The \textbf{connected components} of a digraph $A$ are the maximal connected induced subgraphs of $A$.
\end{definition}

\begin{lemma}
\label{L-dint} If $B$ is an $n\times n$  symmetrizable integer matrix, then  we can choose $D$ in \eqref{D:symmetrizable} to have positive square-roots of integers for all its diagonal entries.
\end{lemma}

\begin{proof} 
Writing $B=(b_{ij})$, $D=\diag(d_1,\dots,d_n)$ and $S=(s_{ij})$, we have from \eqref{E:bd2} that
\begin{equation}\label{E:symmcomp}
d_j^2=(b_{ji}/b_{ij})d_i^2 \textrm{ when }b_{ij}\ne 0\,.
\end{equation}
Thus for all indices $i$, $k$ in the same connected component
of $B$, we see by considering a chain of such identities that $d_i^2/d_k^2$ is rational. 
Thus on fixing some $i$ in this component, for an appropriate positive integer $N$ we can  scale by $N/d_i$ all the $d_k$ in this component (i.e., replace $d_k$ by $Nd_k/d_i$) to make them all square-roots of integers.
Doing this for all connected components of $B$ makes all the $d_i^2$ positive integers.
The relation $D^{-1}BD=S=(s_{ij})$ is preserved by this scaling: $s_{ij} = d_i^{-1}b_{ij}d_j$, and either $s_{ij}=b_{ij}=0$ or $i$ and $j$ lie in the same connected component, and we see that the scaling preserves $s_{ij}$.
\end{proof}

The above proof works with the connected components of $B$.
We wish to dig deeper into the structure of $B$, and break these components up in such a way that the $d_i$ are constant on each piece.

\begin{definition}\label{D:symcomp}
Let $B$ be a symmetrizable $n\times n$ matrix.
Define $B^*$, a symmetric $n\times n$ matrix, by
\[
(B^*)_{ij} = \begin{cases}
b_{ij} & \textrm{if } b_{ij}=b_{ji}\,, \\
0 & \textrm{if } b_{ij}\ne b_{ji}\,.
\end{cases}
\]
(Thus we set to zero any entries of $B$ that were revealing asymmetry of $B$.
If $B$ is symmetric then $B^*=B$.)
The \textbf{symmetric components} of $B$ are defined to be the induced subgraphs of $B$ corresponding to the connected components of $B^*$.
\end{definition}

\begin{lemma}\label{L:symcomp}
Suppose that $B$ is a symmetrizable matrix, with $D^{-1}BD$ symmetric, where $D = \diag(d_1,\dots,d_n)$ (each $d_i>0$).
Then
\begin{itemize}
    \item[(i)] the $d_i$ are constant on the symmetric components of $B$;
    \item[(ii)] let $C_1$, $C_2$ be any two distinct symmetric components of $B$; then for $i\in C_1$, $j\in C_2$ with $b_{ij}\ne 0$, the ratio $b_{ji}/b_{ij}$ is independent of $i$ and $j$.
\end{itemize}
\end{lemma}

\begin{proof}
The first part is immediate from \eqref{E:bd2}.
Then using \eqref{E:bd2} again, along with (i), the second part follows.
\end{proof}

\begin{lemma}\label{L:balancing}
An $n\times n$ matrix $B=(b_{ij})$ is symmetrizable if and only if there exist positive real numbers $d_1$, \dots, $d_n$ such that the balancing condition \eqref{E:bd2} holds.
An $n\times n$ \textbf{integer} matrix $B = (b_{ij})$ is symmetrizable if and only if there exist positive real numbers $d_1$, \dots, $d_n$ such that $d_i^2$ is an integer for each $i$, and the balancing condition \eqref{E:bd2} holds.
\end{lemma}

\begin{proof}
The balancing condition \eqref{E:bd2} is equivalent to the existence of a diagonal matrix $D = \diag(d_1,\dots,d_n)$ such that $D^{-1}BD$ is symmetric.
For the integer case, use also Lemma \ref{L-dint}.
\end{proof}

\begin{definition}
An $n\times n$ real matrix $B$ is said to satisfy the \textbf{cycle condition} if 
\begin{equation}\label{E:cycle}
b_{i_1i_2}b_{i_2i_3} \cdots b_{i_{t-1}i_t}b_{i_ti_1} = b_{i_2i_1}b_{i_3i_2} \cdots b_{i_ti_{t-1}}b_{i_1 i_t}
\end{equation}
holds for all sequences  $i_1$, $i_2$, \dots, $i_t$ of elements of $\{1, \dots, n\}$.
\end{definition}

\begin{lemma} \label{L-cycle}
Any symmetrizable matrix satisfies the cycle condition.
\end{lemma}
\begin{proof}
Suppose that $B$ is an $n\times n$ symmetrizable matrix.
Take any $i_1$, $i_2$, \dots, $i_t\in\{1, \dots, n\}$.
Multiplying all the equations (\ref{E:bd2}) together for $(i,j) = (i_1,i_2)$, $(i_2,i_3)$, \dots, $(i_{t-1}, i_t)$, $(i_t,i_1)$, and then dividing by the (nonzero) product of all the $d_{i_j}^2$ we get \eqref{E:cycle}.
\end{proof}

Thus any symmetrizable matrix is sign symmetric and satisfies the cycle condition.
It is a beautiful well-known fact (see, for example, \cite[Corollary 15.15]{Carter2005}) that these two conditions together are sufficient for a matrix to be symmetrizable.

\begin{proposition} \label{P-spec}An $n\times n$ real matrix $B$ is symmetrizable if and only if it is sign symmetric and satisfies the cycle condition.
\end{proposition}

\begin{proof}
We 
know from Lemmas \ref{L:signsymmetry} and \ref{L-cycle}
that any symmetrizable matrix is sign symmetric and satisfies the cycle condition.

Now suppose that $B=(b_{ij})$ satisfies both these conditions.
For simplicity suppose that $B$ is connected (else treat each component separately).
Set $d_1=1$.
For each neighbour $i$ of vertex $1$, sign symmetry gives both $b_{1i}$ and $b_{i1}$ non-zero, so that we can define 
$
d_i = \sqrt{b_{i1}/b_{1i}}
$.
Then the balancing condition holds when $j=1$ (if any $b_{1i}$ in \eqref{E:bd2} is zero, then both sides are zero).
Next for neighbours $k$ of neighbours $i$ of $1$, define $d_k$ by $d_k = d_i\sqrt{b_{ki}/b_{ik}}$.
By the cycle condition, any vertex $k$ for which $d_k$ has been defined more than once will have received the same value each time.
The balancing condition now holds for $j=1$ and for $j$ any neighbour of $1$.
And so on, we grow our labelling to all the vertices (consistently, thanks to \eqref{E:cycle}), and produce positive numbers $d_i$  such that \eqref{E:bd2} holds.
By Lemma \ref{L:balancing}, $B$ is symmetrizable.
\end{proof}

The above proof also yields a method to test the cycle condition in practice.
On each component, attempt to compute all the $d_i^2$ by the above process.
If no conflicts are found (and having labelled all vertices, push the process one step more to check any edges not yet processed), then the cycle condition holds.

It was noted earlier that the eigenvalues of a symmetrizable matrix are all real.
We record here that the analogue of Cauchy's interlacing theorem \cite{F} holds for symmetrizable matrices.

\begin{theorem}\label{T-inter} Let $B$ be a real $n\times n$ symmetrizable matrix. 
Then the eigenvalues of every  principal $(n-1)\times (n-1)$ submatrix of $B$ are real and interlace with the eigenvalues of $B$: if $B$ has eigenvalues $\lambda_1\le \cdots\le \lambda_n$, and some choice of principal submatrix has eigenvalues $\mu_1\le\cdots\le\mu_{n-1}$, then
\[
\lambda_1\le\mu_1\le\lambda_2\le\mu_2\le\cdots\le\mu_{n-1}\le\lambda_n\,.
\]
Furthermore, these submatrices are also symmetrizable. 
\end{theorem}

\begin{proof}

Take $D$ as in Definition \ref{D:symmetrizable}, so that $S=D^{-1}BD$ is symmetric and has the same eigenvalues as $B$ (all real).

If $B_i$ is obtained by deleting row $i$ and column $i$ from $B$, and similarly $D_i$ is obtained from $D$ and $S_i$ from $S$, then we have $S_i = D_i^{-1}B_iD_i$, showing that $B_i$ is symmetrizable.
The symmetric matrix $S_i$ has the same eigenvalues as $B_i$.
By Cauchy interlacing, the eigenvalues of $S_i$ interlace with those of $S$, and hence the eigenvalues of $B_i$ interlace with those of $B$.
\end{proof}
Interlacing for real symmetrizable matrices was (first?) proved by Kouachi \cite{K}. 
It was Kouachi's result 
that stimulated us to embark on the work described in this paper.

\begin{corollary}\label{C:removal} 
Let $A$ be an $n\times n$ symmetrizable matrix  having all its eigenvalues at $-2$ and $2$. 
Then any induced subgraph of $A$ having all its eigenvalues in the open interval $(-2,2)$ has at most $\lfloor\frac{n}{2}\rfloor$ vertices.
\end{corollary}

\begin{proof}
Since $A$ has at least $\lfloor\tfrac12(n+1)\rfloor$ eigenvalues at $-2$ or at least  $\lfloor\tfrac12(n+1)\rfloor$ eigenvalues at $2$,  removal of up to $\lfloor\tfrac12(n-1)\rfloor$ vertices leaves at least one eigenvalue at either $-2$ or $2$. 
\end{proof}

A symmetrizable integer matrix is said to be \textbf{cyclotomic} if all its eigenvalues are in the interval $[-2,2]$.
We note that the maximal connected cyclotomic integer symmetric matrices of \cite[Theorems 1, 2, 3]{MS2007} remain maximal amongst the larger set of symmetrizable matrices.

\begin{proposition}\label{P:maximal}
If $A$ is a maximal connected cyclotomic integer symmetric matrix, then $A$ is also maximal in the set of connected cyclotomic symmetrizable integer matrices.
\end{proposition}

\begin{proof}
Define the norm of the $i$th row/vertex of a matrix/digraph $C$ to be the $i$th diagonal entry of $C^2$.
(If $C = (c_{kl})$, this is $\sum_j c_{ij}c_{ji}$.)
From \cite[Theorems 1, 2, 3]{MS2007} we know that each row in $A$ has norm $4$.
If $B$ were a connected symmetrizable integer matrix properly containing $A$ as an induced submatrix, then some row of $B$ would have norm greater than $4$ (else $A$ would form a connected component of $B$).
Then the matrix $B^2$ would have some diagonal entry greater than $4$; and note that $B^2$ is symmetrizable (if $D^{-1}BD$ is symmetric, then so is $(D^{-1}BD)^2$).
By repeated application of Theorem \ref{T-inter}, the symmetrizable matrix $B^2$ would have some eigenvalue greater than $4$.
Hence $B$ would have some eigenvalue outside $[-2,2]$, so $B$ would not be cyclotomic.
\end{proof}

 \begin{lemma}\label{L-Eqsym} Suppose that the digraphs $A$ and $B$ are equivalent. Then so are $A^\T$ and $B^\T$.
 \end{lemma}
 
 \begin{proof}
 If $B=\pm P^\T AP$, then $B^\T = \pm P^\T A^\T P$.
 \end{proof}
 
\begin{lemma}\label{L-equiv} 

Each of the following nonsymmetric graphs is equivalent to its transpose:
$\widetilde{A}_{1}'$,
$O_4'$,
${S}_8^{-}$,
$L_4$,
${L}_{n}'\,(n\ge 4$, even$)$, 
${A}_{2}^{\pm }$,
${O}_{4}^{\pm }$,
$\widetilde{C}_n'\,(n\ge 3)$,
$B_2$, $F_4$,
$G_2$,
$O_4''$,
$B_2^{\pm}$.
\end{lemma}

The proof for each of the digraphs listed is demonstrated by showing that for each of these matrices $A$, the transposed matrix $A^\T$ can be transformed back to the original matrix by possibly first replacing $A^\T$ by $-A^\T$, and then by relabelling the vertices, followed by a judicious choice of sign switching (Definition \ref{D:equivalence}). 
For example, $O_4'$ corresponds to the matrix
\[
A = \begin{pmatrix}
0 & -1 & 1 & 0 \\
-1 & 0 & 0 & 1 \\
3 & 0 & 0 & 1 \\
0 & 3 & 1 & 0
\end{pmatrix}\,,
\]
and for the signed permutation matrix
\[
P = \begin{pmatrix}
0 & 0 & -1 & 0 \\
0 & 0 & 0 & 1 \\
-1 & 0 & 0 & 0 \\
0 & 1 & 0 & 0
\end{pmatrix}
\]
one checks that $P^{-1} A^\T P = A$.
Other cases are left to the reader.

In the other direction, we need sometimes to show that two similar-looking digraphs are {\it not} equivalent.

\begin{definition}
In a digraph $G$, we define a {\bf weight modulus sequence} to be a sequence of moduli of edge weights along an induced path in $G$ (an induced subgraph that is a path).
\end{definition}
\begin{lemma}\label{L-inequiv}
If $G_1$ and $G_2$ are digraphs and there is a weight modulus sequence in $G_1$ that does not appear in $G_2$, then $G_1$ and $G_2$ are not equivalent.
 \end{lemma}

\begin{proof}
Any equivalence preserves the set of all weight modulus sequences.
\end{proof}

\begin{corollary}\label{C-inequiv} None of the following digraphs is equivalent to its transpose: $M_n\,(n\ge 2)$, ${B_n}\,(n\ge 3)$, $\widetilde{B}_n\,(n\ge 3)$, $\widetilde{C}_n\,(n\ge 2)$, $\widetilde{F}_4$,  $\widetilde{G}
_2$, $L_n\,(n\ge 6$, even$)$,  ${L}_{n}^{+}\,(n\ge 3$, odd$)$.

Furthermore, the graphs $L_n$ and $L_n'$ ($n\ge 4$, even) are not equivalent.
\end{corollary}

\begin{proof}
To show that each of $M_n$, $B_n$, $\widetilde{B}_n$, $\widetilde{C}_n$, $L_n^+$ is not equivalent to its transpose we apply Lemma \ref{L-inequiv} to the weight modulus sequence $1$, $1$, \dots, $1$, $2$ along a path from one end of the digraph to the other as drawn (right to left for 
$M_n\,(n\ge 3)$, $B_n\,(n\ge 3)$, $L_n^+\,(n\ge 7$, odd); left to right for $\widetilde{B}_n$, $\widetilde{C}_n$). The digraphs $M_2$, $L_3^+$, $L_5^+$ are special cases. For $M_2$, 
 the weight modulus sequence $2$ appears in both $M_2$ and $M_2^\T$, but in one case the weight $2$ is from a neutral vertex to a charged vertex, and in the other case it is not. The same argument applies to $L_3^+$.
For $L_5^+$, note that the weight modulus sequence $1$, $2$ appears in $(L_5^+)^\T$ on an induced path joining the three neutral vertices, whereas in $L_5^+$ a charged vertex is necessarily involved.
For $\widetilde{F}_4$ take the sequence $1$, $2$, $1$, $1$; for $\widetilde{G}_2$ the sequence $3$, $1$; for $L_n$ the sequence $2$, $1$, $1$, \dots, $1$.
To show that $L_n$ and $L_n'$ are not equivalent, for $n\ge 6$ again take the weight modulus sequence $2$, $1$, $1$, \dots, $1$ in $L_n$ from one end of the digraph to the other, and for $n=4$ take the weight modulus sequence $2$, $2$ in $L_4'$.
\end{proof}

\begin{lemma}\label{L:offdiagonal}
Let $\begin{pmatrix} a & b \\ c & d \end{pmatrix}$ be a sign symmetric matrix with both eigenvalues real and in the interval $[-2,2]$.
Then $bc \le 4$.
Moreover if $bc=4$ then $a=d=0$.
\end{lemma}

\begin{proof}
The difference between the two eigenvalues is $\sqrt{(a-d)^2+4bc}$.
This difference is at most $4$, so $4bc \le 16$ (here using that $b$ and $c$ have the same sign), so $bc\le 4$.
Moreover if $bc=4$ then we have $a=d$, and the eigenvalues are $a \pm2$, whence $a=d=0$.
\end{proof}

\section{Proofs of main results}

\begin{proof}[Proof of Theorem \ref{T-1}]

Assume that $A$ is symmetrizable, but not symmetric, is connected and has all its eigenvalues in $[-2,2]$. 
Being symmetrizable, $A$ is sign symmetric, so that $a_{ij}$ and $a_{ji}$ are either both zero, or else both nonzero and of the same sign. 
Hence by nonsymmetry either $A$ or $A^\T$ is equivalent to a digraph with $a_{12}>a_{21}$  both positive integers. 
By repeated application of Theorem \ref{T-inter} we see that the induced subgraph on any two distinct vertices $i$ and $j$ has all its eigenvalues in $[-2,2]$.
So by Lemma \ref{L:offdiagonal} the set $\{a_{12},a_{21}\}$
is equal to either $\{4,1\}$, $\{3,1\}$ or $\{2,1\}$. 
However, we note that if $a_{12}a_{21}=4$ then by Lemma \ref{L:offdiagonal} (again) the digraph $A$ contains $\widetilde{A}_1'$ (Figure \ref{F-AAA}) as an induced subgraph, and the same argument as in Proposition \ref{P:maximal} shows that this is maximal.
So we can assume henceforth that  $\{a_{12},a_{21}\}$ is equal to either $\{3,1\}$ or $\{2,1\}$.

Similarly, we have by Lemma \ref{L:offdiagonal} we have that if $a_{ij}=a_{ji}$ then $a_{ij}=\pm 1$ or $\pm 2$. But if $a_{ij}=a_{ji}=\pm 2$ then $A$ contains $\pm \widetilde{A}_1$ as a subgraph. But this is maximal among integer symmetric matrices having all their eigenvalues in $[-2,2]$, so by Proposition \ref{P:maximal} it 
is maximal among integer symmetrizable matrices having all their eigenvalues in $[-2,2]$. Thus $A=\pm \widetilde{A}_1$,
 contradicting our assumption that $A$ is not symmetric.
 Hence $a_{ij}=a_{ji}=\pm 1$, and $s_{ij}=\pm 1$ by \eqref{E:sij2}.

The symmetrization of $A$, namely $S=D^{-1}AD$ as in Definition \ref{D:symmetrizable}, has the same eigenvalues as $A$, and all entries in $\sqrt{\N0}$.
Indeed all nonzero entries of $S$ are either $\pm 1$, $\pm \sqrt{2}$, or $\pm\sqrt{3}$, and the eigenvalues of $S$ all lie in the interval $[-2,2]$.
Moreover the diagonal entries of $S$ have modulus at most $1$, using Proposition \ref{P:maximal} and \cite[Theorem 3]{MS2007}, along with \eqref{E:sij2}.

Such matrices $S$ were considered by Greaves \cite[Theorems 3.2, 3.4 and 3.5]{G}. 
Greaves gave the form of all maximal such $S$ containing at least one $\pm\sqrt{2}$ or $\pm\sqrt{3}$ entry. 
Those which also satisfy our constraint on the diagonal entries are shown in Figures \ref{F-A13G}, \ref{F-VrG}, \ref{F-WrG}. 
All satisfy $S^2=4I$. 

Now we do not assume that our $A$ is maximal, as we wish to establish the last sentence of the theorem, but we know that $S=D^{-1}AD$ is equivalent to a subgraph of one of the Greaves examples.
A computation as in Lemma \ref{L:equivalent} shows that $A$ is equivalent to a matrix whose symmetrization is a subgraph one of the Greaves graphs, so working up to equivalence we may assume that $S$ is actually a subgraph of one of them.
Moreover $S$ must include at least one irrational edge, else $A$ would be symmetric.

We now determine the $A$ that correspond to such $S$, using \eqref{E:sij2} and \eqref{E:sgnsij} to constrain the possibilities.
When $s_{ij}\in\{0,-1,1\}$ we have $a_{ij}=a_{ji}$.
When $s_{ij} = \sqrt{3}$, $-\sqrt{3}$, $\sqrt{2}$, or $-\sqrt{2}$, we have $\{a_{ij},a_{ji}\}=\{1,3\}$, $\{-1,-3\}$, $\{1,2\}$, or $\{-1,-2\}$ respectively.
Thus an edge
\[
\begin{xy}
@={(0,0)="a", (20,0)="b"}, @@{*\frm<3pt>{*}},
(10,0) *{\sqrt{a}},
"a";(6,0) **@{-}, (14,0);"b" **@{-}
\end{xy}
\]
in $S$, with $a\in\{1,2,3\}$, corresponds to either

\[
\begin{xy}
@={(0,0)="a", (20,0)="b", (40,0)="c", (60,0)="d"}, @@{*\frm<3pt>{*}},
"a";"b" **@{-}, "c";"d" **@{-},
(10,2) *[@!-90]\txt\tiny{$1$}, (10,-2) *[@!90]\txt\tiny{$a$},
(50,2) *[@!-90]\txt\tiny{$a$}, (50,-2) *[@!90]\txt\tiny{1},
(30,0) *\txt{or}
\end{xy}
\]
in $A$, and similarly for negative irrational edges.
Moreover by Lemma \ref{L:symcomp} we must make consistent choices for edges between the same pair of symmetric components of $A$.
In nearly all cases this immediately shows that $A$ is a subgraph of one of the examples claimed: $\widetilde{A}_1'$, $O_4'$, $S_8^-$, $L_n$, $L_n'$, $L_n^\T$, $O_4^{\pm}$, $A_2^{\pm}$, $L_n^+$, $(L_n^+)^\T$.

The most complicated case is when $S = L_{4\G}$, or a subgraph of it, for then Lemma \ref{L:symcomp} places no restrictions on the choice of edges, so \textit{a priori} there are $16$ possibilities to check.
For $S = L_{4\G}$ we use the cycle condition \eqref{E:cycle} to limit the possibilities to  
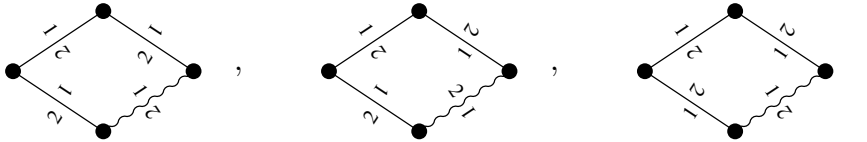
\[
\begin{xy}
@={(0,8)="a", (12,0)="d", (12,16)="b", (24,8)="c", (42,8)="e", (54,0)="h", (54,16)="f", (66,8)="g", (84,8)="i", (96,0)="l", (96,16)="j", (108,8)="k"},
@@{*\frm<3pt>{*}},
@i @={(30,8), (72,8), (114,8)}, @@{*{,}},
"a";"b" **@{-}, "b";"c" **@{-}, "c";"d" **@{~}, "d";"a" **@{-},
"e";"f" **@{-}, "f";"g" **@{-}, "g";"h" **@{~}, "h";"e" **@{-},
"i";"j" **@{-}, "j";"k" **@{-}, "k";"l" **@{~}, "l";"i" **@{-},
@i @={(5,13.5), (17,5.5), (47,13.5),  (89,13.5), (101,5.5)},
@@{*[@!-56.5]\txt\tiny{1}},
(59,5.5) *[@!-56.5]\txt\tiny{2},
@i @={(6.5,10.5), (18.5,2.5), (48.5,10.5), (90.5,10.5), (102.5,2.5)},
@@{*[@!123.5]\txt\tiny{2}},
(60.5,2.5) *[@!123.5]\txt\tiny{1},
@i @={(7,5.5), (49,5.5), (19,13.5)}.
@@{*[@!-123.5]\txt\tiny{1}},
@i @={(61,13.5), (91,5.5), (103,13.5)},
@@{*[@!-123.5]\txt\tiny{2}},
@i @={(60,10.5), (90,2.5), (102,10.5)},
@@{*[@!56.5]\txt\tiny{1}},
@i @={(5.5,2), (47.5,2), (17.5,10)},
@@{*[@!56.5]\txt\tiny{2}},
\end{xy}
\]
or the transposes of these, and working up to equivalence one checks that the distinct possibilities for $A$ are $L_4$ and $L_4'$.

\end{proof}  

\begin{proof}[Proof of Corollary \ref{C-1}] 

Note that the hypothesis here merely requires $A$ to be sign symmetric, rather than symmetrizable.
But apart from $\widetilde{A}_n$ the digraphs we seek cannot contain any cycles (for the spectral radius of $\widetilde{A}_n$ is $2$, and by Perron-Frobenius theory any connected graph containing it as a (not necessarily induced) subgraph has strictly larger spectral radius). 
Hence they all (including $\widetilde{A}_n$) satisfy \eqref{E:cycle}, and so, by Proposition \ref{P-spec} must be symmetrizable. 

Since the symmetric case is covered in \cite[Theorem 9]{MS2007}, we can assume that our matrix (digraph) $A$ is nonsymmetric but symmetrizable.
From Theorem \ref{T-1}, we know that $A$ is a subgraph of one of the digraphs listed there. 
We must look for nonsymmetric, symmetrizable connected subgraphs of these digraphs that are maximal with the property that they are equivalent to a digraph having no negative edges or any negative charges. 
Thus we must remove as few vertices as possible to attain that aim. 
Trivially, there are no such subgraphs from $\widetilde{A}_{1}'$ (other than $\widetilde{A}_1'$ itself),  ${A}_{2}^{\pm }$ or $O_{4}^{\pm }$ (for this last  one, note that one must either remove both the negatively charged vertices or both the positively charged ones).  
From $O_4'$ we remove one vertex to get $\widetilde{G}_{2}$ or its transpose, while from $S_8^{-}$ we remove three vertices to get a digraph equivalent to $\widetilde{F}_{4}$ , $\widetilde{C}_4$ or $\widetilde{B}_3$ 
or the transpose of one of these (all $4$-cycles in $S_8^-$ must be broken). 

To apply the same kind of argument to $L_n$, $L_n^\T$, $L_n'$, $L_n^{+}$ and $(L_n^{+})^\T$, we note that these digraphs contain quadrilaterals with one or three negative edges. 
This number of negative edges may flip between $1$ and $3$ under equivalence, but is never even, and so never zero.  
Such quadrilaterals have an eigenvalue at $2$ or at $-2$. 
Hence, by interlacing, so does any digraph containing such a quadrilateral as a subgraph. 
Thus the only subgraphs we are interested in are those having none of these `odd' quadrilaterals as subgraphs. 
We must therefore remove at least one vertex from every `odd' quadrilateral. 
For $L_n^+$ we must also remove a vertex from the triangles containing the two positively charged vertices.
One way to do this is to remove the bottom row of $r$ vertices from  $L_n$, $L_n^\T$, $L_n'$, $L_n^{+}$ and $(L_n^{+})^\T$. 
These are clearly maximal, as adding back any of these bottom vertices will produce an odd quadrilateral, or, in the $L_n^+$ case, a triangle.
So from $L_{2r+2}$ we obtain $(\widetilde{C}_{r+1})^{\,\T}$,  from $L_{2r+2}^\T$ we obtain $\widetilde{C}_{r+1}$, from  $L_n'$ we obtain $\widetilde{C}_{r+1}'$, from $L_{2r+1}^{+}$ we obtain $M_{r+1}$ and from $(L_{2r+1}^{+})^\T$ we obtain $M_{r+1}^\T$. 
A second way to obtain the kind of graphs we require is to remove from $L_n$ and $L_n'$  the leftmost vertex as well as all bottom vertices except the first. 
This gives the graphs $\widetilde{B}_{r+1}$, $\widetilde{B}_{r+1}^{\,\T}$. 
(Doing this for $L_n^{+}$ gives the symmetric graph $I_{r+1}$ 
of Figure \ref{F-Mr}.)
We leave it as an exercise for the reader to check that these are the only graphs of the kind we are looking for that we can obtain from $L_n$, $L_n'$, $L_n^{+}$ and their transposes.

\end{proof}

\begin{remark}An alternative proof of Corollary \ref{C-1} comes from the identification of symmetrizable non-negative integer matrices as quotients of equitable partitions of graphs \cite{MSwip}.
Again we dispose of cycles and reduce to the symmetrizable case.
Hence we may suppose that $G$ is a quotient of a connected graph $H$.
Now each eigenvector of $G$ lifts to an eigenvector of $H$ that is constant on the subsets of the partition, and in the non-negative case one can apply Perron-Frobenius theory to deduce that $G$ and $H$ have the same spectral radius, so that $H$ also has all its eigenvalues in the interval $[-2,2]$.
Thus $H$ is an induced subgraph of one of the Smith graphs $\widetilde{A}_n$, $\widetilde{D}_n$, $\widetilde{E}_6$, $\widetilde{E}_7$, or $\widetilde{E}_8$.
One can therefore identify all the possible $G$ by considering all possible such $H$, and all possible ways of forming equitable partitions.
\end{remark}

\begin{proof}[Proof of Theorem \ref{T-2}] 
Here we are looking for connected symmetrizable integer matrices with all eigenvalues in the open interval $(-2,2)$. 
Using Corollary \ref{C:removal}, the spectral results in the finite case follow from the spectral results in the affine case by taking all the maximal digraphs of Theorem \ref{T-1}, and removing sets of vertices, and their incident edges, so that the resulting graphs have no eigenvalues at $2$ or $-2$.

Note that we must remove at least half of the vertices of the graph we are considering (Corollary \ref{C:removal}) and do it in such a way that the resulting induced subgraph is still both connected and nonsymmetric. 
This is not possible for $\tilde{A}_{1}'$ or $\tilde{A}_{2}^\pm$.
For $O_4^\pm$ the only possibility is $B_2^\pm$ (or something equivalent to it, or its transpose: henceforth we omit such comments).
For $O_4'$ the only possibility is $G_2$. 
For $S_8^{-}$: (i) we can remove four vertices on a face bounded by two `12' edges (for instance, its front face), giving $O_4''$; (ii) we can remove one vertex and its three adjacent vertices, giving $\widetilde{B}_3$, but $\widetilde{B}_3$ has $2$ as an eigenvalue, so a further vertex then needs to be removed, leaving $B_3$ or $C_3=B_3^\T$; (iii) we can take a $4$-vertex path in $S_8^{-}$ where the three edges are mutually orthogonal, giving $B_4$, $C_4$
or $F_4$.

For $L_n$ (and similarly for $L_n^\T$ and $L_n'$) where $n=2r+2$, we must remove $r+1$ vertices. 
We cannot remove both end vertices, as then the resulting subgraph will be symmetric. 
On the other hand, we cannot leave both end vertices in place, as then removal of $r+1$ vertices produces a nonconnected graph. 
Thus we must remove one end vertex, say the right-most one.
 
Noting that  $L_n$ consists of $r$ `vertical pairs' of vertices, as well as the two end vertices, suppose that we remove the $k\ge0$ rightmost vertical pairs from $L_n$, but not the $(k+1)$st vertical pair, producing a graph $H$ which is a subgraph of $L_{n-2k}$ and has $2(r-k)+1$ vertices.
Thus by Corollary \ref{C:removal} a further $r-k$ vertices must be removed from $H$ before it can possibly have all eigenvalues in $(-2,2)$. 
Now suppose we remove some vertical pair from $H$, but not its rightmost one. 
This would produce a disconnected graph, since the leftmost vertex and a vertex in its rightmost pair would now be in different connected  components. 
Hence to obtain a connected graph with all eigenvalues in $(-2,2)$ we must remove one vertex from each vertical pair. 
Thus we need to remove at least $1+2k+(r-k)=1+r+k$ vertices in all. 
The remaining graph is essentially $B_{r-k}$ or $C_{r-k}=B_{r-k}^\T$, although it may have some negative edges. 
Any negative edges can be made positive under equivalence by sign switching, giving a graph equivalent to $B_{m}$ for some $m\ge 2$.
 
The argument for $L_n^{+}$ or its transpose is similar. 
We cannot remove the leftmost vertex, as the subgraph would then be symmetric. 
Also, the induced graph on the two charged vertices has an eigenvalue at $2$, so at least one of these vertices must be removed. 
If both are removed, the resulting subgraph is a subgraph of $L_n$, and 
we are in the previous case.
Thus we may assume that our subgraph contains a path from the leftmost vertex to a charged vertex. 
By sign switching, we see that we can, under equivalence, assume that all edges of this path are positive. 
On labelling the vertices of the path from left to right as $1$, $2$, \dots, $r+1$ we see that its adjacency matrix has eigenvector $(1,2,\dots,2)$ with eigenvalue $2$. 
Hence, by interlacing, no subgraph of $L_n^{+}$ containing this path can 
 have all its eigenvalues in $(-2,2)$; thus $L_n^{+}$ gives no new graphs.
 
We leave as an exercise for the reader to show that the graphs produced, i.e., $B_2^{\pm}$, $G_2$, $O_4''$, 
$F_4$, $B_n$ do indeed have all their eigenvalues in $(-2,2)$.
 
\end{proof}

\begin{proof}[Proof of Corollary \ref{C-5}] 
This follows almost immediately from Theorem \ref{T-2}.  
One only has to note that removal of a vertex from $O_4''$ or $B_{2}^{\pm }$ gives no digraphs that are not equivalent to either the matrix $(1)$ or subgraphs of $F_4$.
\end{proof}

\begin{ack}  We are very grateful for the referee's meticulous reading of the submitted paper, and for pointing out many minor errors in it.
\end{ack}

\end{document}